\documentclass[12pt,reqno]{amsart}
\usepackage{amsmath, amsthm, amssymb}

\usepackage{comment}
\usepackage{color}
\usepackage{amscd}

\advance \topmargin by -\headheight
\advance \topmargin by -\headsep
     
\evensidemargin \oddsidemargin
\newtheorem{theorem}{Theorem}[section]
\newtheorem{lemma}[theorem]{Lemma}

\newtheorem{proposition}[theorem]{Proposition}

\theoremstyle{definition}

\theoremstyle{remark}

\numberwithin{equation}{section}

\def\bfa{{\mathbf a}}

\def\bfp{{\mathbf p}}

\def\bfx{{\mathbf x}}
\def\bfy{{\mathbf y}}
\def\bfz{{\mathbf z}}

\def\calC{{\mathcal C}} 

\def\calD{{\mathcal D}}

\def\calS{{\mathcal S}}

\def\R{{\mathbb R}}
\def\Z{{\mathbb Z}}\def\Q{{\mathbb Q}}

\def\alp{{\alpha}}

\def\del{{\delta}} \def\Del{{\Delta}}

\def\tet{{\theta}}  
\def\kap{{\kappa}}
\def\lam{{\lambda}}

\def\bfome{{\boldsymbol \ome}}

\def\ome{{\omega}} 
\def\d{{\partial}}
\def\eps{\varepsilon}

\def\d{{\,{\rm d}}}

\def\meas{{\rm meas}}

\newcommand{\id}{\mathbf{1}}

\def\bfometil{{\widetilde{\bfome}}}

\newenvironment{blue}{\color{blue}}{}

\specialcomment{comblue}{\begin{blue}}{\end{blue}}
\excludecomment{comblue}

\DeclareMathOperator{\supp}{supp}
\DeclareMathOperator{\dist}{dist}
\DeclareMathOperator{\Hess}{Hess}

\begin{document}

\author{Damaris Schindler}

\title[Generic ternary diagonal forms]{Diophantine inequalities for generic ternary diagonal forms}

\date{\today}

\maketitle

\begin{abstract}
Let $k\geq 2$ and consider the Diophantine inequality
$$|x_1^k-\alp_2 x_2^k-\alp_3 x_3^k| <\tet.$$
Our goal is to find non-trivial solutions in the variables $x_i$, $1\leq i\leq 3$, all of size about $P$, assuming that $\tet$ is sufficiently large. We study this problem on average over $\alp_3$ and generalize previous work of Bourgain on quadratic ternary diagonal forms to general degree $k$. 

\end{abstract}

\section{Introduction}

In this note we study Diophantine inequalities involving ternary diagonal forms. We fix some degree $k\geq 2$ and let $\alp_2,\alp_3\in \R_{>0}$. Let $\tet >0$ and consider the inequality
\begin{equation}\label{eqn1}
|x_1^k-\alp_2 x_2^k-\alp_3 x_3^k| <\tet.
\end{equation}
The goal is to understand when this inequality has a non-trivial solution if the variables $x_i$ are allowed to range over a box of size $P$ (centered away from the origin in the case $k\geq 3$) and $\tet$ is a function growing with $P$. We will study this problem on average over $\alp_3$ in a bounded interval.\par
In the case $k=2$, Bourgain \cite{BouA16} has shown that one may take $\tet \gg P^{-2/5+\eps}$ for a generic choice of $\alp_3$ (while $\alp_2$ is considered fixed). Assuming the Lindel{\"o}f hypothesis for the Riemann zeta function, he improved this result further to $\tet\gg P^{-1+\eps}$, which is essentially optimal.\par
In the setting of general $k$ one may in an ideal case hope to find a solution to equation (\ref{eqn1}) with $\max_i|x_i|\sim P$, if $\tet \gg P^{k-3}$.\par
Our idea is to combine recent work of Huang \cite{Huang} with the strategy developed by Bourgain \cite{BouA16} (see also work of Blomer, Bourgain, Radziwill and Rudnick \cite{BBRR}) and classical ideas of Titchmarsh \cite{Tit34}, to understand the generalization of the problem to higher degree $k$. Write 
$$f_{\alp_2,\alp_3}(\bfx):= x_1^k-\alp_2 x_2^k-\alp_3 x_3^k.$$
Our main result is the following.

\begin{theorem}\label{thm1}
Let $\alp_2>0$ and $k\geq 3$ be fixed. Then for almost all $\alp_3\in [1/2,1]$ (with respect to the Lebesgue measure) the following statements hold. 
(a) Assuming the Lindel{\"o}f hypothesis for the Riemann zeta function one has
$$\min_{\substack{\bfx\in\Z^3\\ \max_i|x_i|\sim P}}|f_{\alp_2,\alp_3}(\bfx)|\ll P^{k-3+\eps},$$
for any positive $\eps$. Here the implied constant may depend on both $\alp_2$ and $\alp_3$.\\
(b) Unconditionally one has
$$\min_{\substack{\bfx\in\Z^3\\ \max_i|x_i|\sim P}}|f_{\alp_2,\alp_3}(\bfx)|\ll P^{k-12/5+\eps},$$
for any positive $\eps$. Again, the implied constant may depend on both $\alp_2$ and $\alp_3$.
\end{theorem}

Note that the statement in part (a) of Theorem \ref{thm1} is essentially optimal. Indeed, if the values of the form $x_1^k-\alp_2 x_2^k-\alp_3x_3^k$ were equidistributed, one would expect to find $\sim\tet P^{3-k}$ solutions to (\ref{eqn1}) when the variables $x_i$, $1\leq ~i\leq ~3$, are restricted to a box of size $P$. I.e. in order to expect at least one point on average, one needs that $\tet \gg P^{k-3}$. The methods, that we use, are going to produce this heuristic expectation as a main term and hence seem to be restricted to the case that $\tet \gg P^{k-3}$. Similarly, one may consider the inhomogeneous problem
$$|x_1^k-\alp_2 x_2^k-\alp_3 x_3^k-\xi| <\tet,$$
where $|\xi|\leq P^k$ is fixed, and ask for solutions with $\max_i|x_i|\leq P$. If $\tet<~10^{-2}P^{k-3}$ then this equation has no solution in $\max_i|x_i|\leq P$ for a positive proportion of real numbers $\xi$ with $|\xi|\leq P^k$.\\

We compare Theorem \ref{thm1} with pointwise bounds in $\alp_3$ that can be obtained from work of Beresnevich, Dickinson and Velani \cite{BDV} or alternatively Huang \cite{HuangC}. Indeed, the inequality (\ref{eqn1}) can be reformulated as a problem of finding rational points near the planar curve $\mathcal{C}\subset \R^2$ given by $1-\alp_2y_2^k-\alp_3y_3^k=0$. Then finding solutions to (\ref{eqn1}) can be translated into studying a counting function of the type
$$N_\calC (P,\del) = \sharp\left\{\frac{\bfp}{q}\in \Q^2: 1\leq q\leq P: \dist\left(\frac{\bfp}{q},\calC\right)\ll \del/q\right\},$$
with $\del=\frac{\tet}{P^{k-1}}$. In adjusting the support for $x_1,x_2,x_3$ in (\ref{eqn1}), the nonvanishing curvature condition in \cite{BDV} can be met and then Theorem 6 in \cite{BDV} implies that
$$N_{\calC}(P,\del)\gg \del P^2,$$
if $\del\gg P^{-1}$ and $\del P\rightarrow \infty$. Hence, their results are applicable and lead to a non-trivial solution with $\max_i |x_i|\sim P$ to equation (\ref{eqn1}) as soon as $\tet P^{-k+2}\rightarrow \infty$, or in other words they lead to the conclusion that
$$\min_{\substack{\bfx\in\Z^3\\ \max_i|x_i|\sim P}}|f_{\alp_2,\alp_3}(\bfx)|\ll P^{k-2+o(1)}.$$
Note that the finer asymptotic formula in \cite[Theorem 1]{HuangC} has the same requirement to produce at least one solution to (\ref{eqn1}) with $\max_i|x_i|\sim P$.\par
In contrast to finding one solution to equation (\ref{eqn1}) for fixed parameters $\alp_2$ and $\alp_3$, our strategy is to follow ideas of Bourgain \cite{BouA16} and study the problem on average over $\alp_3$. A mean square argument in the background will then lead to another counting problem in more variables that can be attacked with the results from Diophantine approximation. In fact, recent work of Huang \cite{Huang} is just strong enough for our applications. We will make use of the full strength of his results and note that earlier work of Beresnevich, Vaughan, Velani and Zorin \cite{BVVZ} would lead to a weaker version of our Theorem \ref{thm1}.\par
Moreover, in the unconditional case we will make use of good upper bounds on a Dirichlet series of the shape
$$\sum_{\substack{A_1<x_1<B_1\\ A_2< x_2<B_2}}(x_1^k-\alp_2x_2^k)^{it},$$
where $A_1,A_2,B_1,B_2$ are suitable large parameters (typically of size $P$) and $|t|>P^2$. We will have the freedom to choose the support of $x_1$ and $x_2$ in a way to avoid any singularities in the sum. Then we mimic ideas of Titchmarsh \cite{Tit34} on the Epstein zeta function (see also recent work of Blomer \cite{Blo} where Titchmarsh's ideas are generalized to non-diagonal quadratic forms) for the binary degree $k$ form $x_1^k-\alp_2x_2^k$ in place of a positive definite quadratic form. As in \cite{Tit34}, these bounds are based on a version of van der Corput's method.
\par
Finally we remark that the nearly optimal bound in Theorem \ref{thm1} (a) can be obtained unconditionally if one averages simultaneously over $\alp_2$ and $\alp_3$ in bounded intervals, similarly as done in \cite[Section 5]{BouA16}. We sketch the argument in our situation in the last section of this paper.\\

{\em Acknowledgements:} I would like to thank Zeev Rudnick for his generous support and comments on earlier drafts of this paper. \\

{\bf Notation:} We use the notation $\ll$ and $O$ and $o$ with the usual meaning. The implied constants may depend on the fixed parameter $\alp_2$, but are independent of $\alp_3, P$ and $\tet$ (unless stated otherwise). For $x\in \R$, we write $\Vert x \Vert$ for the minimal distance of $x$ to the nearest integer. 

\section{Reduction to a counting problem}
First we choose suitable weight functions for the variables $x_1,x_2,x_3$. Fix real parameters $0<a_i<b_i$ for $1\leq i\leq 3$. For $1\leq i\leq 3$ let $0\leq \ome_i\leq 1$ be smooth bumpfunctions with $\ome_i=1$ on $[\tfrac{1}{2}a_i,\tfrac{3}{4}b_i]$ and $\supp (\ome_i)\subset [\tfrac{1}{4}a_i,b_i]$. Furthermore, let $0\leq \ome_0\leq 1$ be a smooth bumpfunction with $\ome_0=1$ on $[-1,1]$ and $\supp(\ome_0)\subset [-2,2]$ and $\ome_0(t)=\ome_0(-t)$. Moreover, we assume that
\begin{equation}\label{eqn5}
\begin{split}
& \left(\frac{1}{4}a_1\right)^k - \alp_2b_2^k >0\\
&\left(\frac{1}{4}a_1\right)^k < \alp_2 \left(\frac{1}{4}a_2\right)^k + \frac{1}{2}\left(\frac{1}{4}a_3\right)^k\\
&\alp_2b_2^k+b_3^k <b_1^k.
\end{split}
\end{equation}
Note that we consider $\alp_2$ as fixed. Hence the system of equations (\ref{eqn5}) is soluble, if for example we first choose $a_1$, then $b_2$ small enough and then $a_2<b_2$ and then $a_3,b_3$ and finally $b_1$. Define the weight function
$$\bfome (\bfx):=\prod_{i=1}^3 \ome_i(x_i),$$
and write $\bfx$ for the vector $(x_1,x_2,x_3)$. Let $P$ be a large real parameter. We seek a lower bound for the sum
\begin{equation}\label{eqn12}
\sum_{\bfx\in \Z^3}\bfome\left(\frac{\bfx}{P}\right) \id_{[|x_1^k-\alp_2 x_2^k-\alp_3 x_3^k| <\tet]}.
\end{equation}
If $\tet< P^k$, $\alp_3\in [\tfrac{1}{2},1]$ and $P^{-1}\bfx$ lies in the support of $\bfome$ and (\ref{eqn5}) holds, then the inequality
$$| \log (x_1^k-\alp_2x_2^k)-\log (\alp_3x_3^k)|<\frac{\tet}{P^k}$$
implies that
$$|x_1^k-\alp_2 x_2^k-\alp_3 x_3^k| \ll \tet.$$
Hence it is enough to seek a lower bound for the sum
$$S_1(\alp_3):= \sum_{\bfx\in \Z^3}\bfome\left(\frac{\bfx}{P}\right) \id_{[|\log(x_1^k-\alp_2 x_2^k)-\log(\alp_3 x_3^k)| <c_0P^{-k}\tet]},$$
for a small enough constant $c_0$ (only depending on $\alp_2$ and the support of $\bfome$). Set 
$$T:=\frac{2P^k}{c_0\tet}.$$
We define the exponential sums
$$F_1(t):= \sum_{x_1,x_2\in \Z}\ome_1\left(\frac{x_1}{P}\right)\ome_2\left(\frac{x_2}{P}\right)e^{it\log (x_1^k-\alp_2 x_2^k)},$$
and
$$F_2(t):= \sum_{x_3\in \Z} \ome_3\left(\frac{x_3}{P}\right)e^{it\log x_3}.$$
We set
\begin{equation}\label{eqn17}
S_2(\alp_3):=\frac{1}{T} \int_{\R}\hat{\ome_2}\left(\frac{t}{T}\right) F_1(t)  F_2(-kt) e^{-it\log \alp_3}\d t.
\end{equation}
Then we have $S_1(\alp_3)\geq S_2(\alp_3)$. We split the Fourier transform of $\ome_2$ into two parts 
$$\hat{\ome_2}\left(\frac{t}{T}\right)=\hat{\ome_2}\left(\frac{t}{P^{1/2}}\right)+\left(\hat{\ome}_2\left(\frac{t}{T}\right)-\hat{\ome}_2\left(\frac{t}{P^{1/2}}\right)\right).$$

Define 
\begin{equation}\label{eqn18}
S_3(\alp_3):=\frac{1}{T} \int_{\R}\hat{\ome_2}\left(\frac{t}{P^{1/2}}\right) F_1(t)F_2(-kt) e^{-it\log \alp_3}\d t,
\end{equation}
and
\begin{equation}\label{eqn19}
S_4(\alp_3):=\frac{1}{T} \int_{\R}\left(\hat{\ome}_2\left(\frac{t}{T}\right)-\hat{\ome}_2\left(\frac{t}{P^{1/2}}\right)\right) F_1(t) F_2(-kt) e^{-it\log \alp_3}\d t,
\end{equation}
such that $S_2(\alp_3)=S_3(\alp_3)+S_4(\alp_3)$. We first give a lower bound for $S_3(\alp_3)$. 

\begin{lemma}\label{lem1}
There is a constant $c_2>0$ that only depends on $\alp_2$ and the support of $\bfome$ such that for every $\alp_3\in [1/2,1]$ one has
$$ S_3(\alp_3)\geq c_2 P^{3-k}\tet.$$
\end{lemma}

\begin{proof}
We have
\begin{equation*}
S_3(\alp_3)\geq \frac{P^{1/2}}{T}\sum_{\bfx\in \Z^3} \bfome (\bfx)\id_{[|\log(x_1^k-\alp_2 x_2^k)-\log(\alp_3 x_3^k)| <P^{-1/2}]}.
\end{equation*}
Hence there exists a constant $1>c_1>0$ (depending on $\alp_2$ and the support of $\bfome$) such that
\begin{equation*}
S_3(\alp_3)\geq \frac{P^{1/2}}{T}\sum_{\bfx\in \Z^3} \bfome (\bfx)\id_{[|x_1^k-\alp_2 x_2^k-\alp_3 x_3^k| < c_1 P^{k-1/2}]}.
\end{equation*}
We recall the constraints of the support of $\bfome$ given by the system of equations (\ref{eqn5}) and we recall that $\alp_3\in [1/2,1]$. Hence, after fixing $x_2$ and $x_3$ in the support of $\ome_2(P^{-1}x_2)$ resp. $\ome_3(P^{-1}x_3)$, there are at least $\gg P^{1/2}$ choices for for $x_1$ that satisfy the inequality
$$|x_1^k-\alp_2 x_2^k-\alp_3 x_3^k| < c_1 P^{k-1/2}.$$
This leads to the lower bound
$$S_3(\alp_3)\gg \frac{P^{1/2}}{T}P^{2+1/2}\gg P^{3-k}\tet.$$
Here the implied constant only depends on $\alp_2$ and the support of $\bfome$. 
\end{proof}

If we consider $S_3(\alp_3)$ as the main term of the counting function $S_2(\alp_3)$, then Lemma \ref{lem1} shows us that we should only expect to count at least one solution to the inequality (\ref{eqn1}) if $\tet\gg P^{k-3}$.\par
Now we turn to estimates for $S_4(\alp_3)$ (on average), which we consider as an error term. As in \cite{BouA16}, we use Chebyshev's inequality to bound
\begin{equation}\label{eqn20}
\meas\{\alp_3\in [1/2,1]: |S_4(\alp_3)|\geq c_2 P^{3-k}\tet\}\leq c_2^{-2}\tet^{-2}P^{2k-6} \int_{1/2}^1|S_4(\alp_3)|^2\d \alp_3.
\end{equation}
Here $c_2$ is the constant given by Lemma \ref{lem1}. 
Set 
$$I_1:=\int_{1/2}^1|S_4(\alp_3)|^2\d \alp_3.$$
By the symmetry of the weight function $\ome_2$ we have
$$\left|\hat{\ome}_2\left(\frac{t}{T}\right)-\hat{\ome}_2\left(\frac{t}{P^{1/2}}\right)\right|\leq c_3 \min\left(1,\frac{t^2}{P},\left(\frac{T}{|t|}\right)^{10}\right),$$
for a positive constant $c_3$. By Parseval's theorem we have
$$I_1\leq c_3^2T^{-2}\int_\R \min\left(1,\frac{t^4}{P^2},\left(\frac{T}{|t|}\right)^{20}\right)|F_1(t)|^2|F_2(-kt)|^2\d t.$$
The contribution of values of $t$ in the integral bounding $I_1$ with $|t|<P^{1/10}$, say,  turns out to be negligible. 
Define
$$I_2:= c_2^{-2}\tet^{-2}P^{2k-6}c_3^2T^{-2} \int_{|t|>P^{1/10}} \min\left(1,\left(\frac{T}{|t|}\right)^{20}\right)|F_1(t)|^2|F_2(-kt)|^2\d t.$$
Then one has
\begin{equation}\label{eqn21}
\meas\{\alp_3\in [1/2,1]: |S_4(\alp_3)|\geq c_2 P^{3-k}\tet\}\leq c_4P^{-3/2}+I_2,
\end{equation}
for some positive constant $c_4$. 
The first term in the bound on the right hand side is already sufficiently small for our purposes. Hence our goal now is to bound the integral $I_2$. For this we combine estimates for $F_2(t)$ for not too small values of $t$ with an $L^2$-estimate of the exponential sum $F_1(t)$.\par

Let
\begin{equation}\label{eqn24}
I_3:=\int_{\R} \min\left(1,\left(\frac{T}{|t|}\right)^{10}\right) |F_1(t)|^2\d t.
\end{equation}
We bound $I_2$ by
\begin{equation}\label{eqn27}
I_2\ll P^{-6} \max_{|t|\geq P^{1/10}}\left(\min\left(1,\left(\frac{T}{|t|}\right)\right)|F_2(t)|\right)^{2} I_3.
\end{equation}
We now continue to bound the integral $I_3$. For any $y\in \R$ we define
$$I(y):= \int_{\R} \min\left(1,\left(\frac{T}{|t|}\right)^{10}\right)e^{ity}\d t.$$

\begin{lemma}\label{lemI}
For $y\in \R$ one has 
$$|I(y)|\ll \min \left(\frac{1}{|y|}, T\right).$$
\end{lemma}

\begin{proof}
We can always bound $I(y)$ by
\begin{equation*}
\begin{split}
|I(y)|&\leq \int_{\R}\min\left(1,\left(\frac{T}{|t|}\right)^{10}\right)\d t \\
&\leq 2T+\int_{|t|>T}\left(\frac{T}{|t|}\right)^{10}\d t \\
& \ll 2T+T^{-9+10} \ll T.
\end{split}
\end{equation*}
On the other hand, for $y\neq 0$, and using partial integration we can bound
\begin{equation*}
\begin{split}
I(y)&= \int_{|t|\leq T} e^{ity}\d t + \int_{|t|>T} \left(\frac{T}{|t|}\right)^{10} e^{ity} \d t \\
& \ll \frac{1}{|y|}.
\end{split}
\end{equation*}
\end{proof}

We now introduce the notation $\bfy=(y_1,\ldots, y_4)$ and write 
$$\bfometil (\bfy):= \ome_1(y_1)\ome_2(y_2)\ome_1(y_3)\ome_2(y_4).$$
For $\frac{1}{\log P} \ll U\leq T$ we define
\begin{equation}\label{eqnI4}
I_4(U):= U \sum_{\bfy\in \Z^4}\bfometil \left(\frac{\bfy}{P}\right) \id_{|\log (y_1^k-\alp_2 y_2^k)-\log (y_3^k-\alp_2 y_4^k)|<\frac{1}{U}}. 
\end{equation}
By Lemma \ref{lemI} and a dyadic decomposition we can bound $I_3$ by
$$I_3 \ll \log P \sup_{\frac{1}{\log P}\ll U\leq T} I_4(U).$$
We further bound $I_4(U)$ by
$$I_4(U) \leq U \sum_{\bfy\in \Z^4}\bfometil \left(\frac{\bfy}{P}\right) \id_{|y_1^k-\alp_2 y_2^k-(y_3^k-\alp_2 y_4^k)|\ll \frac{P^k}{U}}.$$
We now bound $I_4(U)$ for small values of $U$ and use in the next section work of Huang \cite{Huang} to treat the case of large values of $U$. For $U\ll P$ we fix an arbitrary choice of values for $y_2,y_3,y_4$ in the support of $\ome_1$, resp. $\ome_2$, and consider the inequality
$$ |y_1^k-\alp_2 y_2^k-(y_3^k-\alp_2 y_4^k)|\ll \frac{P^k}{U}.$$
This restricts $y_1^k$ to an interval of length $\ll \frac{P^k}{U}$. Due to the support of $\ome_1$ the variable $y_1$ is itself of size $\sim P^k$. Hence there are at most $\ll \frac{P}{U}$ choices for $y_1$ and we obtain
\begin{equation}\label{eqn25}
I_4(U)\ll U P^3 \frac{P}{U}\ll P^4.
\end{equation}

\section{A counting problem}
The goal of this section is to give an upper bound for $I_4(U)$ for $c_5P \leq U\leq T$ (for $c_5$ a sufficiently large positive constant) using work of Huang \cite{Huang}. In order to apply his work we need to understand the function $I_4(U)$ as a counting function of rational points close to a hypersurface with non-vanishing Gaussian curvature. Before we give the details of our application, we recall one of the main results in \cite{Huang} that we are going to use. Assume that $\mathcal{S}$ is a piece of a hypersurface in real $n$-space given in the form
$$(\bfz, f(\bfz)), \quad\bfz = (z_1,\ldots, z_{n-1})\in \mathcal{D},$$
where $\calD\subset \R^{n-1}$ is a connected open bounded set and $f$ is a smooth function on $\calD$. Assume that for any point $\bfz\in \calD$ the Hessian $\nabla^2 f(\bfz)$ satisfies
$$C_1< |\det \nabla^2f(\bfz)| < C_2,$$
for positive constants $C_1,C_2$. Then after taking $\calD$ small enough one finds that $\nabla f: \calD \rightarrow \nabla f(\calD)$ is a diffeomorphism. Let $\ome$ be a smooth and compactly supported weight function on $\R^{n-1}$ with support inside $\calD$. For $0<\del <1/2$  define the counting function
$$N_\calS^{\ome}(Q,\del) = \sum_{\substack{\bfa\in \Z^{n-1}\\ q\leq Q \\ \Vert qf(\bfa/q)\Vert <\del}}\ome\left(\frac{\bfa}{q}\right).$$

\begin{theorem}[Theorem 2 in \cite{Huang}]\label{thm3}
Under the above assumptions one has
$$N_\calS^{\ome}(Q,\del) =\frac{2\hat{\ome}(0)}{n}\del Q^n + O_{\calS, \ome}(E_n(Q)),$$
with 
$$E_3(Q)= Q^2 \exp(c\sqrt{\log Q}),$$
and 
$$E_n(Q)= Q^{n-1}(\log Q)^\kap, \quad n\geq 4,$$
for some positive constants $c$ and $\kap$.
\end{theorem}

We now apply this result to our counting function $I_4(U)$. If the inequality 
$$|y_1^k-\alp_2 y_2^k-(y_3^k-\alp_2 y_4^k)|\ll \frac{P^k}{U}$$
holds, then 
$$y_1^k\in [y_3^k+\alp_2y_2^k-\alp_2y_4^k - c_6 \frac{P^k}{U},y_3^k+\alp_2y_2^k-\alp_2y_4^k + c_6 \frac{P^k}{U}],$$
for some positive constant $c_6$. Assume that $P^{-1}y_3$ lies in the support of $\ome_1$ and $P^{-1}y_2$, $P^{-1}y_4$ lie in the support of $\ome_2$. Then, by equation (\ref{eqn5}) one has $y_3^k+\alp_2y_2^k-\alp_2y_4^k\gg P^k$, and hence
$$y_1 \in [(y_3^k+\alp_2y_2^k-\alp_2y_4^k)^{1/k} - c_7\frac{P}{U},(y_3^k+\alp_2y_2^k-\alp_2y_4^k)^{1/k}+ c_7 \frac{P}{U}]\cap \Z,$$
for some positive constant $c_7$. For $c_5$ sufficiently large and $U\geq c_5P $ we see that this interval can contain at most one integer point and if this is the case, then
$$\Vert (y_3^k+\alp_2y_2^k-\alp_2y_4^k)^{1/k}\Vert \ll \frac{P}{U}.$$
Hence, for $c_5$ sufficiently large, we can rewrite the bound for $I_4(U)$ as
$$I_4(U)\leq U \sum_{\substack{(y_2,y_3,y_4)\in \Z^3\\ 4^{-1}a_1 P\leq y_3\leq b_1P\\\Vert (y_3^k+\alp_2y_2^k-\alp_2y_4^k)^{1/k}\Vert \ll \frac{P}{U} }} \ome_2\left(\frac{y_2}{P}\right)\ome_2\left(\frac{y_4}{P}\right).$$
Let $\mathcal{D}:= [(1/4) a_2b_1^{-1},4b_2a_1^{-1}]^2$. For $(z_2,z_4)\in\mathcal{D}$ we define the function 
$$f(z_2,z_4):= (1+\alp_2z_2^k-\alp_2z_4^k)^{1/k},$$
which is well-defined by recalling the relations (\ref{eqn5}). Then there is a smooth weight function $0\leq \ome_4\leq 1$ with $\supp (\ome_4)\subset [(1/4) a_2b_1^{-1},4b_2a_1^{-1}]$ such that 
\begin{equation}\label{eqn23}
\begin{split}
I_4(U)&\leq U \sum_{\substack{(y_2,y_4)\in \Z^2\\ 4^{-1}a_1P\leq y_3\leq b_1P\\\Vert y_3f(y_2/y_3,y_4/y_3)\Vert \ll \frac{P}{U} }} \ome_4\left(\frac{y_2}{y_3}\right)\ome_4\left(\frac{y_4}{y_3}\right)\\
&\leq U \sum_{\substack{(y_2,y_4)\in \Z^2\\ 1\leq y_3\leq b_1P\\\Vert y_3f(y_2/y_3,y_4/y_3)\Vert \ll \frac{P}{U} }} \ome_4\left(\frac{y_2}{y_3}\right)\ome_4\left(\frac{y_4}{y_3}\right).
\end{split}
\end{equation}
We compute the Hessian $\nabla^2 f = \left(\frac{\partial^2 f}{\partial z_i\partial z_j}\right)$ and find that
\begin{equation*}
\begin{split}
\det (\nabla^2f) =& (k-1)^2\alp_2^2(1+\alp_2z_2^k-\alp_2z_4^k)^{2/k-4} \\& \times \left\{-z_2^{k-2}z_4^{k-2}(1-\alp_2z_4^k)(1+\alp_2z_2^k)-\alp_2^2z_2^{2k-2}z_4^{2k-2}\right\}.
\end{split}
\end{equation*}
Again, we recall the first of the conditions in (\ref{eqn5}) and deduce that $\det (\nabla^2f)$ is strictly negative on the set $z_2,z_4$ such that $z_2,z_4\in \supp (\ome_4)$. Hence there are positive constants $c_8$ and $c_9$ such that
$$c_8\leq |\det (\nabla^2f)|\leq c_9,$$
for all $(z_2,z_4)\in \calD$, i.e. the hypersurface determined in its Monge form by $f(z_2,z_4)$ has non-vanishing curvature on the domain $\mathcal{D}$.\par
Now we are in a position that we can apply Theorem \ref{thm3} (after refining the domain $\mathcal{D}$ sufficiently) with the parameters 
$Q=b_1P$ and $\del \ll \frac{P}{U}$. Note that $\del < 1/2$ if we take $c_5$ sufficiently large. We deduce that 
\begin{equation}\label{eqn22}
I_4(U)\ll U(\del P^3+ P^{2+\eps}).
\end{equation}
In fact, the factor $P^\eps$ could be replaced by $\exp(c\sqrt{\log P})$ for a positive constant $c$. We summarize our resulting bound for $I_3$ in the following proposition.

\begin{proposition}\label{prop1}
Assume that (\ref{eqn5}) holds. Then we have
$$I_3\ll_\eps P^{4+\eps}+\frac{P^{k+2+\eps}}{\tet},$$
for any $\eps >0$.
\end{proposition}

\section{Proof of Theorem \ref{thm1}}

From equation (\ref{eqn27}) we recall the following bound on $I_2$
\begin{equation*}
I_2\ll P^{-6} \max_{|t|\geq P^{1/10}}\left(\min\left(1,\left(\frac{T}{|t|}\right)\right)|F_2(t)|\right)^{2} I_3.
\end{equation*}
In fact, our function $F_2(t)$ coincides with the one in \cite{BouA16} (after adjusting the smooth weight function). In particular equation (2.6) in \cite{BouA16} shows that
$$F_2(t)\ll P^{1/2} \int_{-\infty}^\infty \frac{|\zeta(\frac{1}{2}+i(y-t))|}{1+|y|^{10}}\d y.$$
Hence, assuming the Lindel{\"o}f hypothesis for the Riemann zeta function, we get the bound
$$F_2(t)\ll P^{1/2}(1+|t|)^\eps.$$
Together with equation (\ref{eqn27}) and Proposition \ref{prop1} we deduce for $\tet >1$ that 
\begin{equation*}
\begin{split}
I_2&\ll_\eps P^{-6}P^{1+\eps} \left(P^{4+\eps}+\frac{P^{k+2+\eps}}{\tet}\right)\\
&\ll P^{-1+\eps}+\frac{P^{k-3+\eps}}{\tet}.
\end{split}
\end{equation*}
Unconditionally we have $|\zeta(\frac{1}{2}+it)|\ll (1+|t|)^{1/6}$ and hence
$$F_2(t)\ll P^{1/2}|t|^{1/6}.$$
A combination of equation (\ref{eqn27}) and Proposition \ref{prop1} then leads to the bound
\begin{equation*}
\begin{split}
I_2&\ll_\eps P^{-6}P\left(\frac{P^k}{\tet}\right)^{1/3} \left(P^{4+\eps}+\frac{P^{k+2+\eps}}{\tet}\right)\\
&\ll P^{-1+k/3+\eps}\tet^{-1/3}+\frac{P^{(4/3)k-3+\eps}}{\tet^{4/3}}.
\end{split}
\end{equation*}
We combine these estimates with the bound in equation (\ref{eqn21}) and Lemma \ref{lem1} to deduce the following Proposition. 

\begin{proposition}\label{prop2}
Let $\alp_2>0$ and $k\geq 3$ be fixed. Let $P$ be a large real parameter and assume that $\tet > P^{k-3}$. 

(a) Assuming the Lindel{\"o}f hypothesis for the Riemann zeta function, the inequality 
$$|x_1^k-\alp_2x_2^k-\alp_3x_3^k|<\tet$$
has a non-trivial solution with $\max_i|x_i|\sim P$ for all $\alp_3\in [1/2,1]$ with the exception of a set of $\alp_3$ of measure at most
$$\ll P^{-1+\eps}+\frac{P^{k-3+\eps}}{\tet}.$$
(b) The same statement holds unconditionally with an exceptional set of $\alp_3\in [1/2,1]$ of measure at most
$$\ll P^{-1+k/3+\eps}\tet^{-1/3}+\frac{P^{(4/3)k-3+\eps}}{\tet^{4/3}}.$$
\end{proposition}

Theorem \ref{thm1}(a) follows now from Proposition \ref{prop2}(a) via the Borel-Cantelli lemma and a dyadic decomposition of all possible values of $P$ (see for example section 8 in \cite{BBRR}). Part (b) of Proposition \ref{prop2} would lead to an unconditional result with the exponent $k-12/5+\eps$ in Theorem \ref{thm1}(b) replaced by the weaker bound $k-9/4+\eps$. In order to obtain the stronger bound in Theorem \ref{thm1}(b) we follow the ideas of Bourgain in \cite[Section3 + Section 4]{BouA16}. Indeed it turns out that the same formalism works in our case provided we can establish the following two main ingredients that correspond to Lemma 2 and Lemma 3 in \cite{BouA16}. The first of the results that we need is the following mean square bound.

\begin{lemma}\label{lem2}
Let $F_1(t)$ be defined as before and let $T>P^2$. Assume that equation (\ref{eqn5}) holds. Then one has
$$\meas\{|t|\leq T: |F_1(t)|>\lam\}\ll TP^{2+\eps}\lam^{-2}.$$
\end{lemma}

\begin{proof}
It is enough to establish the bound
$$\int_{|t|\leq T}|F_1(t)|^2\d t \ll P^{2+\eps} T.$$
The integral on the left hand side amounts to
$$\sum_{\substack{x_1,x_2\in \Z\\ y_1,y_2\in \Z}}\ome_1\left(\frac{x_1}{P}\right)\ome_2\left(\frac{x_2}{P}\right)\ome_1\left(\frac{y_1}{P}\right)\ome_2\left(\frac{y_2}{P}\right)\int_{|t|\leq T} e^{it (\log (x_1^k-\alp_2x_2^k)-\log(y_1^k-\alp_2y_2^k))}\d t.$$
Using the same ideas as in Lemma \ref{lemI} and a dyadic decomposition for the values of $\log (x_1^k-\alp_2x_2^k)-\log(y_1^k-\alp_2y_2^k)$ one can bound this expression by
$$\log P \sup_{\frac{1}{\log P}\ll U\leq T} I_4(U),$$
with $I_4(U)$ defined as in equation (\ref{eqnI4}). Hence by the bounds for $I_4(U)$ in equation (\ref{eqn25}) and equation (\ref{eqn22}) we obtain 
$$\int_{|t|\leq T}|F_1(t)|^2\d t\ll P^{4+\eps}+T P^{2+\eps},$$
which is what is required, since we assumed $T>P^2$.

\end{proof}

The second ingredient that is needed for the refinement in part (b) of Theorem \ref{thm1} is motivated by partial sum bounds of the Epstein zeta function in the case $k=2$. It turns out that the same bounds can be performed for the sum defining $F_1(t)$ although the quadratic form is replaced by a binary degree $k$ form. 

\begin{lemma}\label{lem3}
Assume that $|t|>P^2$ and assume that (\ref{eqn5}) holds. Then one has
$$|F_1(t)|\ll P|t|^{1/3+\eps}.$$
\end{lemma}

Our proof of Lemma \ref{lem3} follows the strategy of Titchmarsh \cite{Tit34} on bounding partial sums of the Epstein zeta function using van der Corput differencing and second derivative tests, see also a generalization of his arguments to non-diagonal quadratic forms by Blomer \cite{Blo}.

\begin{proof}
By partial summation it is enough to establish the bound
$$S(X_1,X_2):=\sum_{\substack{(1/4)a_1P\leq x_1<X_1\\ (1/4)a_2P\leq x_2<X_2}}e^{it\log (x_1^k-\alp_2x_2^k)}\ll P|t|^{1/3+\eps},$$
for $X_1\leq b_1P$ and $X_2\leq b_2P$. Note that the first equation in (\ref{eqn5}) ensures that $x_1^k-\alp_2x_2^k\sim P^k$ on the range of summation that is considered. Moreover, the required bound for $S(X_1,X_2)$ holds trivially if $(X_1-\tfrac{1}{4}a_1P)(X_2-\tfrac{1}{4}a_2P)\ll P|t|^{1/3}$ or $|t|\gg P^3$. Hence we may assume that $(X_1-\tfrac{1}{4}a_1P)(X_2-\tfrac{1}{4}a_2P)\gg P|t|^{1/3}$ and $|t|<P^3$. Choose a parameter $\rho := \lfloor P|t|^{-1/3}\rfloor$ and note that $\rho\leq \min\{X_1-\tfrac{1}{4}a_1P, X_2-\tfrac{1}{4}a_2P\}$. Set $f(x_1,x_2)=t \log (x_1^k-\alp_2x_2^k)$. We apply \cite[Lemma $\beta$]{Tit34} and find that
\begin{equation}\label{eqn28}
S(X_1,X_2)\ll \frac{P^2}{\rho} + \frac{P}{\rho}\left(\sum_{\mu=1}^{\rho-1}\sum_{\nu=0}^{\rho-1}|S_1|\right)^{1/2} + \frac{P}{\rho}\left(\sum_{\mu=1}^{\rho-1}\sum_{\nu=0}^{\rho-1}|S_2|\right)^{1/2} ,
\end{equation}
with 
$$S_1=\sum_{(1/4)a_1P\leq x_1\leq X_1-\mu}\sum_{(1/4)a_2P\leq x_2\leq X_2-\nu} e^{i(f(x_1+\mu,x_2+\nu)-f(x_1,x_2))},$$
and
$$S_2=\sum_{(1/4)a_1P\leq x_1\leq X_1-\mu}\sum_{(1/4)a_2P+\nu\leq x_2\leq X_2} e^{i(f(x_1+\mu,x_2-\nu)-f(x_1,x_2))}.$$
We now focus on estimating $S_1$, while the bounds for $S_2$ will follow in exactly the same way. For a fixed choice of $\mu,\nu$ we write $\lam = \max\{\mu,\nu\}$ and we set
$$g(x_1,x_2):= f(x_1+\mu,x_2+\nu)-f(x_1,x_2).$$
For $x_1,x_2$ in the range of summation, a direct calculation shows that
$$\frac{\partial^2 g}{\partial x_1^2}(x_1,x_2),\ \frac{\partial^2 g}{\partial x_1\partial x_2}(x_1,x_2),\ \frac{\partial^2 g}{\partial x_2^2}(x_1,x_2)\ll \frac{t\lam}{P^3},$$
where the implicit constant may depend on $\alp_2$. As in \cite{Tit34} we divide the range of summation in $S_1$ into squares $\Del_{p,q}$ of side length $l$ with 
$$l\leq \frac{c_{10}P^3}{t\lam},$$
with a sufficiently small constant $c_{10}$. Note that $1\ll l\ll P$ by the assumption that $P^2<|t|< P^3$. We choose $c_{10}$ small enough such that on each square $\Del_{p,q}$ we can find integers $Q_1$ and $Q_2$ such that the function
$$G(x_1,x_2)= g(x_1,x_2)-2\pi Q_1 x_1-2\pi Q_2 x_2$$
satisfies
$$|\partial_{x_1}G(x_1,x_2)|, |\partial_{x_2}G(x_1,x_2)|\leq \frac{3\pi }{2}$$
on the square $\Del_{p,q}$. Hence an application of \cite[Lemma $\gamma$]{Tit34} shows that 
$$S_1= \sum_{p,q}\int_{\Del_{p,q}}e^{iG(x_1,x_2)}\d x_1 \d x_2 + O\left(\frac{P^2}{l}\right).$$
The goal now is to apply a second derivative test to each of the integrals over a box $\Del_{p,q}$. Set $\Phi = x_1^k-\alp_2x_2^k$. Using a Taylor expansion on a fixed such box we compute
\begin{equation*}
\begin{split}
\det (\Hess (G(x_1,x_2)))=& 2\alp_2k^3(k-1)t^2\mu^2\Phi^{-3} x_1^{2k-4}x_2^{k-2}\\
&-2\alp_2k^2(k-1)(k-2)t^2\mu\nu\Phi^{-2}x_1^{k-3}x_2^{k-3} \\
&- 2 \alp_2^2k^3(k-1)t^2\nu^2\Phi^{-3}x_1^{k-2}x_2^{2k-4} +O_{\alp_2,k}\left(\lam^3 t^2 P^{-7}\right).
\end{split}
\end{equation*}
On the support of the variables $x_1,x_2$ one then finds that the quadratic form in $\mu,\nu$ on the right hand side of this expression is bounded below by
$$\det (\Hess (G(x_1,x_2)))\gg t^2\lam^2P^{-6}.$$
To see this, one can for example explicitly compute the zeros of the quadratic form depending on $x_1,x_2$.\par
An application of \cite[Lemma $\delta$ and Lemma $\zeta$]{Tit34} now leads to the bound
$$S_1\ll \frac{t^2\rho^2}{P^4}\frac{P^{3+\eps}}{t\lam}+\frac{t\lam}{P}\ll \frac{\rho^2 t^{1+\eps}}{\lam P}.$$
From equation (\ref{eqn28}) we deduce that
\begin{equation*}
\begin{split}
S(X_1,X_2)&\ll \frac{P^2}{\rho} + \frac{P}{\rho} \left(\sum_{\lam=1}^{\rho -1} \frac{\rho^2 t^{1+\eps}}{ P}\right)^{1/2}\ll \frac{P^2}{\rho}+P^{1/2}|t|^{1/2+\eps}\rho^{1/2}
\\&\ll P|t|^{1/3+\eps}.
\end{split}
\end{equation*}
This is exactly the bound that we required. 
\end{proof}

Now the same calculations as in \cite[Section 4]{BouA16} lead to the following proposition. 

\begin{proposition}\label{prop3}
Let $\alp_2>0$ and $k\geq 3$ be fixed. Let $P$ be a large real parameter and assume that $\tet > P^{k-3}$. 

Then the inequality 
$$|x_1^k-\alp_2x_2^k-\alp_3x_3^k|<\tet$$
has a non-trivial solution with $\max_i|x_i|\sim P$ for all $\alp_3\in [1/2,1]$ with the exception of a set of $\alp_3$ of measure at most
$$\ll P^{\frac{5}{6}k-2+\eps}\tet^{-\frac{5}{6}}+P^{\frac{10}{9}k-\frac{8}{3}+\eps}\tet^{-\frac{10}{9}}.$$
\end{proposition}
Again, Theorem \ref{thm1} (b) now follows via the Borel-Cantelli lemma and a dyadic decomposition for all possible values of $P$.

\section{Averaging over both $\alp_2$ and $\alp_3$}

If one averages in equation (\ref{eqn1}) over both parameters $\alp_2$ and $\alp_3$ simultaneously, then one obtains the essentially optimal statement in Theorem \ref{thm1} (a) unconditionally. 

\begin{theorem}\label{thm2}
Let $k\geq 3$ be fixed. Then for almost all $(\alp_2,\alp_3)\in [1/2,1]^2$ (with respect to the Lebesgue measure) one has
$$\min_{\substack{\bfx\in\Z^3\\ \max_i|x_i|\sim P}}|f_{\alp_2,\alp_3}(\bfx)|\ll P^{k-3+\eps},$$
for any positive $\eps$. Here the implied constant may depend on both $\alp_2$ and $\alp_3$.
\end{theorem}

Note that the same statement holds for $k=2$ and is proved (up to a different sign of $\alp_2$) in \cite[Section 5]{BouA16}. We observe that these arguments can be generalized to higher degree $k$ and now give a sketch of the proof of Theorem \ref{thm2}.\par
First one chooses a replacement of the weight function $\bfome(\bfx)$ which is independent of $\alp_2$, e.g. in replacing $\alp_2$ in (\ref{eqn5}) by $1/2$ or resp. $1$. One then defines $S_2$ (we now omit the dependence on $\alp_3$ in the notation, similarly as the dependence on $\alp_2$ is implicitly understood) as in equation (\ref{eqn17}). We continue to write $S_2=S_3+S_4$ with 
\begin{equation*}
S_3:=\frac{1}{T} \int_{\R}\hat{\ome_2}\left(\frac{t}{P^{1/2}}\right) F_1(t)F_2(-kt) e^{-it\log \alp_3}\d t,
\end{equation*}
and
\begin{equation*}
S_4:=\frac{1}{T} \int_{\R}\left(\hat{\ome}_2\left(\frac{t}{T}\right)-\hat{\ome}_2\left(\frac{t}{P^{1/2}}\right)\right) F_1(t) F_2(-kt) e^{-it\log \alp_3}\d t.
\end{equation*}
Then, as in Lemma \ref{lem1}, for any fixed choice of $(\alp_2,\alp_3)\in [1/2,1]^2$ one obtains the lower bound
$$S_3\geq c_2P^{3-k}\tet.$$
Here $c_2$ is now independent of $\alp_2$ (as we consider a compact interval for $\alp_2$). By Chebyshev's inequality one has
\begin{equation*}
\begin{split}
\meas\{(\alp_2,\alp_3)\in& [1/2,1]^2: |S_4|\geq c_2 P^{3-k}\tet\}\\ &\leq c_2^{-2}\tet^{-2}P^{2k-6} \int_{1/2}^1\int_{1/2}^1|S_4|^2\d \alp_3\d\alp_2.
\end{split}
\end{equation*}
The same arguments that before lead to equation (\ref{eqn21}) now give the bound
\begin{equation}\label{eqn31}
\meas\{(\alp_2,\alp_3)\in [1/2,1]^2: |S_4|\geq c_2 P^{3-k}\tet\}\leq c_4P^{-3/2}+\tilde{I}_2,
\end{equation}
with
\begin{equation*}
\tilde{I}_2\ll  P^{-6} \int_{|t|>P^{1/10}} \min\left(1,\left(\frac{T}{|t|}\right)^{20}\right)\left(\int_{1/2}^1 |F_1(t)|^2\d\alp_2\right) |F_2(-kt)|^2\d t.
\end{equation*}

We now observe that \cite[Lemma 4]{BouA16} generalizes to higher degree $k$.

\begin{lemma}\label{lem4}
For $t\neq 0$ one has
$$\int_{1/2}^1 |F_1(t)|^2\d\alp_2 \ll P^{2+\eps}+\frac{P^{4+\eps}}{|t|}.$$
\end{lemma}

\begin{proof}
We recall that
$$|F_1(t)|^2= \sum_{x_1,x_2,x_3,x_4\in \Z}\prod_{j=1}^2\prod_{l=0}^1\ome_j\left(\frac{x_{j+2l}}{P}\right)e^{it\log (x_1^k-\alp_2 x_2^k)-it\log (x_3^k-\alp_2 x_4^k)}.$$
Next, we perform the integration over $\alp_2$. We compute
$$\frac{\partial}{\partial \alp_2}( \log (x_1^k-\alp_2 x_2^k)-\log (x_3^k-\alp_2 x_4^k)) = \frac{-x_2^k x_3^k+x_4^k x_1^k}{(x_1^k-\alp_2 x_2^k)(x_3^k-\alp_2x_4^k)}.$$
Hence, for $P^{-1}x_1,P^{-1}x_3\in \supp(\ome_1)$ and $P^{-1}x_2,P^{-1}x_4\in \supp(\ome_2)$, we obtain
$$\int_{1/2}^1 e^{it\log (x_1^k-\alp_2 x_2^k)-it\log (x_3^k-\alp_2 x_4^k)}\d \alp_2 \ll \min\left(1, \frac{P^{2k}}{|t||(x_4x_1)^k-(x_2x_3)^k|}\right).$$
This leads us to the bound
\begin{equation*}
\begin{split}
\int_{1/2}^1 |F_1(t)|^2\d\alp_2 &\ll \sum_{\substack{x_1,x_2,x_3,x_4\in\Z\\ x_i\sim P, 1\leq i\leq 4}}\min\left(1, \frac{P^{2k}}{|t||(x_4x_1)^k-(x_2x_3)^k|}\right)\\
&\ll P^\eps \sup_{\frac{P^{2k}}{|t|}\leq U\ll P^{2k}}\frac{P^{2k}R(U)}{|t|U},
\end{split}
\end{equation*}
with
$$ R(U)=\sharp\{x_1,x_2,x_3,x_4\in \Z, x_i\sim P, 1\leq i\leq 4: |(x_4x_1)^k-(x_2x_3)^k|\leq U\}.$$
Note that 
$$R(U)\ll P^\eps \sharp\{z_1,z_2\in \Z, z_1,z_2\sim P^2: |z_1^k-z_2^k|\leq U\}.$$
As the polynomial $z_1^k-z_2^k$ has no repeated roots, we have the level set estimate
$$R(U)\ll P^{2+\eps} \left(1+\frac{U}{P^{2(k-1)}}\right).$$
This gives the bound
\begin{equation*}
\begin{split}
\int_{1/2}^1 |F_1(t)|^2\d\alp_2 &\ll P^{2+\eps} \sup_{\frac{P^{2k}}{|t|}\leq U\ll P^{2k}}\frac{P^{2k}}{|t|U}\left(1+\frac{U}{P^{2(k-1)}}\right)\\
&\ll P^{2+\eps}+\frac{P^{4+\eps}}{|t|}.
\end{split}
\end{equation*}
\end{proof}

As in \cite{BouA16} we can now use dyadic division of the integration domain of $t$ to estimate the integral $\tilde{I}_2$. For $t\sim 2^l$ we combine the bound 
\begin{equation*}
\begin{split}
\int_{|t|\sim 2^l} |F_2(t)|^2\d t &\ll P\int_{|t|\sim 2^l} \int_{-\infty}^\infty \int_{-\infty}^\infty  \frac{|\zeta(\frac{1}{2}+i(y_1-t))|}{1+|y_1|^{10}} \frac{|\zeta(\frac{1}{2}+i(y_2-t))|}{1+|y_2|^{10}}\d y_1\d y_2\d t\\ &\ll P\int_{|t|\sim 2^l} \int_{-\infty}^\infty \int_{-\infty}^\infty  \frac{|\zeta(\frac{1}{2}+i(y_1-t))|^2}{1+|y_1|^{10}} \frac{1}{1+|y_2|^{10}}\d y_1\d y_2\d t\\ &\ll P 2^{l(1+\eps)}.
\end{split}
\end{equation*}
with Lemma \ref{lem4}. This gives the estimate
\begin{equation*}
\begin{split}
\tilde{I}_2&\ll P^{-6}\sum_{l\gg \log P}\min\left(1,\frac{T^{20}}{2^{20l}}\right)\left(P^{2+\eps}+\frac{P^{4+\eps}}{2^l}\right)P2^{l(1+\eps)}\\
&\ll P^{-5+\eps}\left(P^2+\frac{P^4}{T}\right)T \\
&\ll TP^{-3+\eps}+P^{-1+\eps}.
\end{split}
\end{equation*}
We recall that $T\sim \frac{P^k}{\tet}$, and hence
$$\tilde{I}_2\ll P^{k-3+\eps}\tet^{-1}+P^{-1+\eps}.$$
Theorem \ref{thm2} now follows along the same lines as Theorem \ref{thm1} above.



\bibliographystyle{amsbracket}
\providecommand{\bysame}{\leavevmode\hbox to3em{\hrulefill}\thinspace}

\end{document}